\documentclass[12pt]{article}
\usepackage[margin=2.9cm]{geometry}
\usepackage[cp866]{inputenc}
\usepackage[english]{babel}
\usepackage{a4}
\usepackage{amsfonts,amssymb,amsmath,amsthm,euscript, xspace,bbm}
\usepackage{secdot,textcomp}
\usepackage[bookmarksopen,colorlinks,linkcolor=blue]{hyperref}
\usepackage{graphicx}
\usepackage{authblk}
\usepackage[dvipsnames]{xcolor}
\definecolor{electricultramarine}{rgb}{0.25, 0.0, 1.0}

\sectiondot{subsection}

\title{Uniform cross-$t$-intersecting families: proving Hirschorn's conjecture up to polynomial factor\thanks{Some results of this paper in a preliminary form were obtained by the second and the third author in~\cite{kozachinskiy2019application}.}} 

\author[2]{Georgii P. Bulgakov}
\author[1]{Alexander Kozachinskiy}
\author[2,3,4]{Mikhail N. Vyalyi}

\affil[1]{University of Warwick, Coventry, UK}
\affil[2]{National Research University Higher School of Economics, Moscow, Russian Federation}
\affil[3]{Moscow Institute of Physics and Technology, Russian Federation}
\affil[4]{Dorodnicyn Computing
  Centre, FRC CSC RAS, Moscow, Russian Federation}

\let\leq\leqslant
\let\geq\geqslant
\let\le\leq
\let\ge\geq

\def\sm{\setminus}

\let\epsilon\varepsilon

\let\phi\varphi

\def\X{{\mathcal X}}
\def\Y{{\mathcal Y}}
\def\F{{\mathcal F}}
\def\G{{\mathcal G}}

\def\X{\ensuremath{{\mathcal X}}}

\def\A{{\mathcal A}}
\def\B{{\mathcal B}}

\def\sm{\setminus}

\newtheorem*{known}{Theorem}

\newtheorem{theorem}{Theorem}
\newtheorem{lemma}[theorem]{Lemma}
\newtheorem{proposition}[theorem]{Proposition}

\newtheorem{corollary}[theorem]{Corollary}
\newtheorem{conj}[theorem]{Conjecture}

\begin{document}
\maketitle

\begin{abstract}
We consider a problem of maximizing the product of the sizes of two uniform cross-$t$-intersecting families of sets. We show that the value of this maximum is at most polynomially larger (in the size of a ground set) than a quantity conjectured by  Hirschorn~\cite{hirschorn2008asymptotic}. At the same time, we observe that it can be strictly bigger. 
\end{abstract}

\section{Introduction}

Numerous works in combinatorics are devoted to extremal problems for families of sets with certain restrictions on their intersections. In this paper we study this kind of problems for \emph{uniform cross-$t$-intersecting} families.

Let $n$ denote the size of a ground set $[n] = \{1, 2, \ldots, n\}$. By a ``family'' we always mean a set $\F$ such that any element of $\F$ is a subset of the ground set $[n]$.
A family $\F$ is called \emph{$m$-uniform} if $|F|=m$ for any set $F$ in the
family. The set of all $m$-element subsets of $[n]$ is denoted by
$\binom{[n]}{m}$. Two families $\F$ and $\G$ are called \emph{cross-$t$-intersecting} if $|F \cap G| \ge t$ for all $F\in\F, G\in \G$.

Let $h$ be some function mapping pairs of families to real numbers. We will refer to such functions as to \emph{functionals}. For a functional $h$ and for 
integers
$a, b$ and $t$ one can consider the following optimization problem:
\begin{subequations}
\begin{alignat}{2}
&\!\max      &\qquad& h(\F, \G) \label{eq:optProb}\\
&\text{subject to} &      & \mbox{$\F$ is an $a$-uniform family},\label{eq:constraint1}\\
&                  &      &  \mbox{$\F$ is a $b$-uniform family},\label{eq:constraint2} \\
&                  &      &  \mbox{$\F$ and $\G$ are cross-$t$-intersecting}.\label{eq:constraint3}
\end{alignat}
\end{subequations}

We will denote the maximum in (\ref{eq:optProb}--\ref{eq:constraint3}) by $N_h(n, a, b, t)$. Two most studied functionals are the product and the sum functionals:
$$h_{prod}(\F, \G) = |\F| \cdot |\G|, \qquad h_{sum}(\F, \G) = |\F| + |\G|.$$
For these two functionals we will use, respectively, the following notation:
$$N_{prod}(n, a, b, t), \qquad N_{sum}(n, a, b, t).$$

The starting point  in this line of research was the famous Erd\H{o}s-Ko-Rado
theorem. It only considers families that are cross-$t$-intersecting with themselves. Such families are simply called \emph{$t$-intersecting} (formally, a family $\F$ is  $t$-intersecting if a pair $(\F, \F)$ is cross-$t$-intersecting). 

\begin{known}[Erd\H{o}s, Ko, Rado, \cite{erdos1961intersection}] There exists a function $n_0(m, t)$ such that for all $1\leq t\leq m$ and $n\geq n_0(m,t)$ the maximal size of a
$t$-intersecting $m$-uniform family is $\binom{n-t}{m-t}$. 
\end{known}

Thus, the maximum in this theorem is attained on a family
$$\left\{ X\in\binom{[n]}{m}: [t]\subseteq X\right\}.$$
Still, Erd\H{o}s-Ko-Rado
theorem is applicable only when $n$ is sufficiently large with respect to $m$ and $t$. The general case was addressed by
Ahlswede and Khachatrian~\cite{AKh97}. They proved that for all $n, m, t$ the maximal size of a 
$t$-intersecting $m$-uniform family  is attained on some family of the form
\[ 
\left\{ X\in\binom{[n]}{m}: |X\cap[t+2i]|\geq t+i
\right\}, \quad 0\leq i\leq \frac{n-t}2\,.
\]

Matsumoto and Tokushige~\cite{MATSUMOTO198990} considered  an analogue of the  Erd\H{o}s-Ko-Rado
theorem for cross-$t$-intersecting pairs of  families. More specifically, they studied the case $t = 1$.  They proved that for $n\ge 2a, n\ge 2b$ we have:
\[
 N_{prod}(n,a,b,1) = \binom{n-1}{a-1}   \binom{n-1}{b-1},
\]
 and thus this quantity is attained on a pair:
$$\F = \left\{F\in\binom{[n]}{a} : 1 \in F\right\}, \qquad \G = \left\{G\in\binom{[n]}{b} : 1 \in G\right\}.$$
Finally, they observed that this pair is not optimal for $b = a - 1, n = 2a - 1$ so that their conditions are sharp.

In~\cite{TOKUSHIGE20101167} Tokushige extended these results from cross-$1$-intersecting families to cross-$t$-intersecting families (but only for the case $a = b$). He proved that for all $0 < p < 0.11, 1 \le t\le 2/p$ there exists $n_0$ such that for all $n> n_0$ and $a/n = p$ we have:
\[
 N_{prod}(n,a,a,t) = \binom{n-t}{a-t}^2.
\]

This shows for a wide range of parameters the same family which establishes a maximum in the Erd\H{o}s-Ko-Rado theorem simultaneously establishes a maximum for  $N_{prod}(n,a, a, t)$.

In~\cite{hirschorn2008asymptotic} Hirschorn conjectured that for some function $n_0(a, b, t)$ we have:
\[
 N_{prod}(n,a,b,t) = \binom{n-t}{a-t}   \binom{n-t}{b-t}
\]
as long as $n>n_0(a,b,t)$. This conjecture was later proved by Borg~\cite{Borg2014TheMP}. As in the Erd\H{o}s-Ko-Rado theorem, we only have this equality for the regime when $n$ is much larger than all the other parameters. As Hirshorn observed, this equality is no longer true in the other settings. In~\cite{hirschorn2008asymptotic} he made one more conjecture, now for all possible values of $n, a, b, t$.

\begin{conj}[\cite{hirschorn2008asymptotic}]\label{Hirschorn} The quantity
  $N_{prod}(n,a, b, t)$ is attained on some pair of families of the form
$$\F = \left\{F \in\binom{[n]}{a} : |F\cap [s]| \ge u\}\right\}, \,\, \G = \left\{G \in\binom{[n]}{b} : |G\cap [s]| \ge v\}\right\},$$ 
where $u, v, s\in [n]$ are such that $u + v = s + t$.
\end{conj}
Conjecture \ref{Hirschorn} can be seen as a generalization of Ahlswede-Khachatrian theorem.
We will  refer to the pairs of families from this conjecture as to \emph{Hirschorn pairs}. 

In the literature there are conditional results that are proved assuming Conjecture \ref{Hirschorn}. For instance, Zvonarev and Raigorodskii~\cite[Theorem 3]{ZR15} derive from Conjecture \ref{Hirschorn} an improvement of the classical results of Frankl and R{\"o}dl~\cite{frankl1987forbidden} on uniform families of sets with forbidden intersections.

Finally, let us mention that some theorems of Erd\H{o}s-Ko-Rado type were proved for tuples of (more than two) pairwise cross-$t$-intersecting families, and not only the product functional but also for the sum functional~\cite{WANG2011455, BORG2014117}.

\subsection{Main result}

The main contribution of this paper is the following result, motivated by Conjecture \ref{Hirschorn}.
\begin{theorem}
\label{th:main}
Let $\F$ be an $a$-uniform family and $\G$ be a $b$-uniform family. Assume that $\F$ and $\G$ are cross-$t$-intersecting. Then there exists an $a$-uniform family $\F^*$ with $|\F^*| = |\F|$ and a $b$-uniform family $\G^*$ with $|\G^*| = |\G|$ such that 
$$\forall (F, G)\in \F^*\times \G^*\,\, \exists s\in [n]  \mbox{ such that } |F\cap[s]| + |G\cap[s]| \ge s + t.$$
\end{theorem}
\begin{proof}
See Section \ref{sec:th:main}.
\end{proof}
Observe that $\F^*$ and $\G^*$ as in Theorem \ref{th:main} must be cross-$t$-intersecting. In fact, Theorem \ref{th:main} simply says that $\F^*\times \G^*$ is a subset of the union of the Cartesian products of the Hirschorn's pairs.

\begin{theorem}[Reformulation of Theorem \ref{th:main}]
\label{th:ref}
Let $\F$ be an $a$-uniform family and $\G$ be a $b$-uniform family. Assume that $\F$ and $\G$ are cross-$t$-intersecting. Then there exists an $a$-uniform family $\F^*$ with $|\F^*| = |\F|$ and a $b$-uniform family $\G^*$ with $|\G^*| = |\G|$ such that
$$\F^* \times \G^* \subseteq \bigcup\limits_{\substack{u,v,s\in [n] \\ u + v = s + t}} \left\{F \in\binom{[n]}{a} : |F\cap [s]| \ge u\}\right\} \times \left\{G \in\binom{[n]}{b} : |G\cap [s]| \ge v\}\right\}.$$
\end{theorem}
This result yields an \emph{approximate} version of Conjecture \ref{Hirschorn}, and not only for the product functional, but for any \emph{symmetric super-additive} functional $h$. Here we call a functional $h$ \emph{symmetric} if $h(\F, \G) = h(\F^\prime, \G^\prime)$ whenever $|\F^\prime| = |\F|, |\G^\prime| = |\G|$. In turn, we call 
we call a functional $h$ \emph{super-additive} if whenever families $\F, \G, \F_1, \G_1, \ldots, \F_k,\G_k$ are such that
$$\F\times \G \subseteq \bigcup\limits_{i = 1}^k \F_i\times \G_i,$$
we have:
$$h(\F, \G) \le \sum\limits_{i = 1}^k h(\F_i, \G_i).$$
For instance, the sum and the product functionals are symmetric and super-additive. Now, let $\widehat{N}_h(n, a, b, t)$ be the maximum of $h(\F,\G)$ over the Hirshcorn's pairs of families, i.e.,
\begin{align*}
&\widehat{N}_h(n, a, b, t) \\ &= \max\limits_{\substack{u,v,s\in [n] \\ u + v = s + t}}h\left( \left\{F \in\binom{[n]}{a} : |F\cap [s]| \ge u\}\right\}, \left\{G \in\binom{[n]}{b} : |G\cap [s]| \ge v\}\right\}\right).
\end{align*}
Obviously, $N_h(n, a, b, t) \ge \widehat{N}_h(n, a, b, t)$ for any $h$, and Conjecture \ref{Hirschorn} states that $N_{prod}(n, a, b, t) = \widehat{N}_{prod}(n, a, b, t)$. We obtain from Theorem \ref{th:ref} the following Corollary.
\begin{corollary} 
\label{sa}
For any symmetric super-additive functional $h$ we have:
$$\widehat{N}_h(n,a,b,t) \le N_h(n,a,b,t) \le n^2 \cdot \widehat{N}_h(n,a,b,t).$$
\end{corollary}
\begin{proof}
Let $\F, \G$ be a pair of families on which $N_h(n,a,b,t)$ is attained. Take $\F^*, \G^*$ as in Theorem \ref{th:ref}. Since $h$ is symmetric, we have $N_h(n, a, b, t) = h(\F,\G) = h(\F^*,\G^*)$. So it remains to bound $h(\F^*,\G^*)$. Now, $\F^*\times \G^*$ belongs to the union of the Cartesian products of the Hirschorn's pairs.  Due to super-additivity of $h$ this means that $h(\F^*,\G^*)$ is bounded by the sum of the values of $h$ on the Hirschorn's pairs. As there are at most $n^2$ Hirschorn's pairs, we are done.
\end{proof}
For most values of the parameters, 
$N_h(n,a,b,t)$ is super-polynomial in $n$ so that the $n^2$-factor is rather small compared to $\widehat{N}_h(n,a,b,t)$ and $N_h(n,a,b,t)$. In fact, for some applications of the Hirschorn's
conjecture polynomial factors are irrelevant. For instance, this applies to a conditional result of Zvonarev and Raigorodskii~\cite[Theorem 3]{ZR15} for which we, thus, obtain an unconditional proof.

Despite this result, it is still interesting whether Hirschorn's conjecture is true exactly, i.e., whether $N_{prod}(n,a,b,t) = \widehat{N}_{prod}(n,a,b,t)$. We disprove this and provide an infinite series of counter-examples in the following proposition.

\begin{proposition}
\label{inf_counter}
For infinitely many $n$ we have $$N_{prod}(n, 2, n - 2, 1) > \widehat{N}_{prod}(n,2,n - 2,1).$$
\end{proposition}
\begin{proof}
See Section \ref{sec:inf_counter}.
\end{proof}
Arguably, this setting of parameters is rather specific. Indeed, $F\in\binom{[n]}{2}$ and $G\in\binom{[n]}{n - 2}$ intersect if and only if $F$ and $G$ are not the complements of each other. 
So obviously  $N_{prod}(n, 2, n - 2, 1)$ is attained on the families obtained in the following way:  split $\binom{[n]}{2}$ into two equals parts and replace one  part by the family consisting of  the complements of the pairs from  this part. 
Here we assume that $\binom{n}{2}$ is even.  All that remains to do is to check that for infinitely many $n$ such that $\binom{n}{2}$ is even no Hirschorn pair partitions $\binom{[n]}{2}$ evenly.

For completeness, we also provide concrete counter-examples for a less specific settings of parameters.
Let $a=2k+1$, $b=2k+2$, $n=4k+3$, $t=2$. Direct computations show that
for $3\leq k\leq 50$ families
\[
\begin{aligned}
  \A_k =& \big\{ A \in\binom{[n]}{a}: |A\cap[2k+1]|\geq k+1 \text{ and
  } |A\cap[2k+3]|\geq k+2 \big\},\\
  \B_k =& \big\{ B \in\binom{[n]}{b}: |B\cap[2k+1]|\geq k+2 \text{ or
  } |B\cap[2k+3]|\geq k+3 \big\}
\end{aligned}
\]
provide a larger value of $|\A_k|\cdot|\B_k|$ than any Hirshorn's
pair. It is worth to mention that   $\widehat N_{prod}(4k+3,2k+1,2k+3,2)$ is attained on a Hirshorn
pair $\F$, $\G$ with $u = v=k+1$, $s = 2k$ for this range of the parameter
$k$. It can be proven that $|\A_k|\cdot|\B_k|>|\F|\cdot |\G|$ for all
values of~$k$. 
Unfortunately, we were not able to prove that $|\F|\cdot |\G|
  =\widehat N_{prod}(4k+3,2k+1,2k+3,2) $ for all $k\geq 3$. If so, it
  would be another  infinite series of counter-examples to Hirshorn's
  conjecture. In any way,  there is still a possibility that for a wide range of parameters Hirschorn's conjecture is true exactly. We leave this question for future work.

\subsection{Asymptotic bounds for $N_{prod}(n, a, b, t)$}
Our result allows to express $N_{prod}(n, a, b, t)$ up to a polynomial factor in terms of the products of the binomial coefficients.
\begin{proposition}
\label{M_prop}
 It holds that
$M \le N_{prod}(n, a, b, t) \le n^3 \cdot M$, where
\begin{equation}
\label{M}
M = \max\limits_{\substack{ u, v, s\in [n] \\ u + v \ge s + t}} \binom{s}{u} \binom{n - s}{a - u} \binom{s}{v} \binom{n - s}{b - v}.
\end{equation}
\end{proposition}
\begin{proof}
From Theorem \ref{th:ref} we obtain:
$$N_{prod}(n, a, b, t) \le \sum\limits_{\substack{ u, v, s\in [n] \\ u + v \ge s + t}} \binom{s}{u} \binom{n - s}{a - u} \binom{s}{v} \binom{n - s}{b - v}.$$
All the terms in this sum are bounded by $M$, and there are at most $n^3$ of them, so we get
$N_{prod}(n, a, b, t)  \le n^3 M$.
On the other hand, $N_{prod}(n,a,b,t) \ge M$, because  $\binom{s}{u} \binom{n - s}{a - u} \binom{s}{v} \binom{n - s}{b - v}$ equals $|\F| \cdot |\G|$ 
 for families $\F = \{F\in\binom{[n]}{a} : |F\cap [s]| = u\}, \G= \{G\in\binom{[n]}{a} : |G\cap [s]| = v\}$ that are cross-$t$-intersecting due to the restriction $u + v \ge s + t$.
\end{proof}

Unfortunately, it seems hard to compute the exact asymptotic behavior of 
the quantity \eqref{M}. We provide two upper bounds on it that might be sufficient for applications.

\begin{proposition}
\label{upper_bounds}
Assume that $a + b < n + t$. Then the following upper bounds hold:
\begin{itemize}
\item \textbf{(a)} 
$N_{prod}(n,a, b, t) \le n^3 \cdot \exp\left\{2h\left(\frac{a + b - 2t}{2n}\right)n\right\}$, where $h\colon [0, 1]\to[0,+\infty)$ denotes the Shannon function:
$$h(x) = x\ln\left(\frac{1}{x}\right) + (1 - x) \ln\left(\frac{1}{1 - x}\right).$$
\item \textbf{(b)}
$N_{prod}(n,a, b, t)\le 8n^4 \cdot \exp\left\{-\frac{(\min\{t, n - a - b + t\})^2}{8 \cdot (\min\{a, n - a\} + \min\{b, n - b\})}\right\} \cdot \binom{n}{a}\cdot\binom{n}{b}$.
\end{itemize}
\end{proposition}
\begin{proof}
See Section \ref{sec:upper_bounds}.
\end{proof}

Note that if $a + b \ge n + t$, then $N_{prod}(n, a, b, t) = \binom{n}{a}\cdot \binom{n}{b}$.

It is not hard to see that the first bound asymptotically outperforms the second one when $a = b$. On the other hand, for some values of the parameters the first bound (unlike the second one) is exponentially worse than a trivial upper bound  $N_{prod}(n,a,b,t) \le \binom{n}{a}\cdot\binom{n}{b}$.

\section{Proof of Theorem \ref{th:main}}
\label{sec:th:main}

\textbf{Notation.} For $X\subseteq[n]$ we denote $\overline{X} = [n]\setminus X$. For $X\subseteq [n]$ and for $1 \le i\le |X| $ by $m(X, i)$ we denote the $i$th element of $X$ in the increasing order. By the \emph{weight} of a set $X\subseteq [n]$ we mean $w(X) = \sum_{i\in X} i$. In turn, by the weight of a family $\X \subseteq 2^{[n]}$ we mean $w(\X) = \sum_{X\in \X} w(X)$.

\bigskip

Intuitively, among all pairs of families $\F^\prime, \G^\prime$, satisfying:
\begin{align}
\label{cond1}
\F^\prime \mbox{ is $a$-uniform},\,\, \G^\prime \mbox{ is $b$-uniform}, \\
\label{cond2}
|\F^\prime| = |\F|,\,\, |\G^\prime| = |\G|, \\
\label{cond3}
\F^\prime, \G^\prime \mbox{ are cross-$t$-intersecting},
\end{align}
we want one ``maximally compressed''  to the left (if one imagines elements of $[n]$ going from left to right in the increasing order). This hopefully will make sets from our families highly concentrated on the initial segments of $[n]$, as required in Theorem \ref{th:main}. We formalize this intuition via a classical compression (a.k.a. shifting) technique~\cite{Frankl87}. Namely, for $i, j\in [n]$ define a function
$$\delta_{ij} \colon 2^{[n]}\to 2^{[n]}, \qquad \delta_{ij}(X) = \left\{
\begin{aligned}
  &(X\sm\{j\})\cup\{i\}, &&\text{if}\ j\in X,\ i\notin X,\\
  &X, &&\text{otherwise.}
\end{aligned}\right.$$
If $i < j$, then $\delta_{ij}$ is called a \emph{left compression}. A left compression removes a larger element from a set and instead adds a smaller element, so intuitively the set becomes more compressed to the left. But we do this only when the larger element is in the set and the smaller element is not so that this operation  preserves the size of the set.

These operations allow us to formalize our intuitive notion of a ``maximally compressed family'' as a family which is closed under any left compression (no left compression maps a set from our family to a set outside the family).

\begin{lemma} 
\label{maximal_compressed}
There exists a pair $(\F^*, \G^*)$ satisfying \textup{(\ref{cond1}--\ref{cond3})} such that $\F^*$ is closed under any left compression.
\end{lemma}
In fact, one can achieve that $\G^*$ is closed under any left compression as well, but we do not need this for the argument.
\begin{proof}[Proof of Lemma \ref{maximal_compressed}]
We define auxiliary compression operations for families of sets. Namely, for a family $\X$ and for $i < j$ define
$$\Delta_{ij}(\mathcal{X}) = \{\delta_{ij}(X) \mid X\in\X, \delta_{ij}(X)\notin \X\}
\cup \{ X \mid X\in\X, \delta_{ij}(X)\in \X\}.$$
Observe that $\Delta_{ij}$ preserves the size of $\X$. Remarkably, $\Delta_{ij}$ also preserves the cross-$t$-intersection property. I.e., if two families $\X$ and $\Y$ are cross-$t$-intersecting, then so are $\Delta_{ij}(\X)$ and $\Delta_{ij}(\Y)$, for any $i < j$, see~\cite[Lemma 2.1]{Borg2014TheMP}.

Now, take $(\F^*, \G^*)$ minimizing the quantity $w(\F^\prime)$ over all $(\F^\prime, \G^\prime)$ satisfying (\ref{cond1}--\ref{cond3}). If $\F^*$ is not closed under some left compression, then $\Delta_{ij}(\F^*) \neq \F^*$ for some $i < j$.
This means that $\Delta_{ij}(\F^*)$ has smaller weight than $\F^*$ (a left compression can only decrease the weight of a set, and whenever the resulting set is different from the original one, the weight becomes strictly smaller). 
 Therefore a pair $\Delta_{ij}(\F^*), \Delta_{ij}(\G^*)$ cannot satisfy (\ref{cond1}--\ref{cond3}). On the other hand, it satisfies  (\ref{cond1}--\ref{cond2}) because $\delta_{ij}$ and $\Delta_{ij}$ preserve the size (correspondingly, of a set and of a family of sets), and it satisfies \eqref{cond3} because $\Delta_{ij}$ preserves the cross-$t$-intersection property.
\end{proof}

It remains to show that for any $\F^*, \G^*$ as in Lemma \ref{maximal_compressed}, for all $F\in\F^*$, $G\in \G^*$ we have:
$$\exists s\in [n]  \mbox{ such that } |F\cap[s]| + |G\cap[s]| \ge s + t.$$
If $n\leq a+b-t$, then this inequality holds for $s=n$. For the rest of parameters we reformulate  this condition on $F$ and $G$  as follows.
\begin{lemma}
\label{eq_border}
Let $F\in\binom{[n]}{a}$ and $G\in\binom{[n]}{b}$. 
If $n> a+b-t$, then the following two conditions are equivalent:
\begin{itemize}
\item \textbf{(a)}  $\exists s\in [n]  \mbox{ such that } |F\cap[s]| + |G\cap[s]| \ge s + t$;
\item \textbf{(b)} $\exists i \in [a - t + 1] \mbox{ such that } m(\overline{G}, i) > m(F, t + i - 1)$ 
(note that the size of $\overline{G}$ is $n - b \ge a - t + 1$ for  the parameters into consideration so that $m(\overline{G}, i)$ is well-defined).
\end{itemize}
\end{lemma}
\begin{proof}
\textbf{Proof of \emph{(a)} $\implies$ \emph{(b)}.} Define $i = |\overline{G}\cap [s]| + 1$. Let us show that $i\in [a - t + 1]$. Indeed, $|\overline{G}\cap [s]| = s - |G\cap [s]| \le |F\cap[s]| + |G\cap[s]| - t -  |G\cap [s]| \le a - t$, because $|F| = a$.

 Now, observe that $m(\overline{G}, i) > s$ (because $[s]$ contains only $i - 1$ elements of $\overline{G}$). On the other hand, $|F\cap[s]| \ge s + t - |G\cap[s]| = t + |\overline{G}\cap [s]| = t + i - 1$. Hence $m(F, t + i - 1) \le s < m(\overline{G}, i)$.

\textbf{Proof of \emph{(b)} $\implies$ \emph{(a)}.}
Set $s= m(F, t + i - 1)$. In $[s]$ there are exactly $t + i - 1$ elements of $F$ and at most $i - 1$ elements of $\overline{G}$. Hence $|F\cap[s]| + |G\cap[s]| \ge t + i - 1 + s - (i - 1) = s + t$, as required.
\end{proof}
Now, assume for contradiction that there are $F\in\F^*$ and $G\in \G^*$ for which the condition \textbf{\emph{(a)}} of Lemma \ref{eq_border} is violated. Among all such pairs $(F, G)$ take one minimizing the weight $w(F)$.
 Since $\F^*$ and $\G^*$ are cross-$t$-intersecting, we have $|F\cap G| \ge t$, or, equivalently, $|F\cap\overline{G}| \le a - t$. This means that there exists $j\in [a - t + 1]$ such that $u = m(\overline{G}, j)$ is not in $F$. Now, since $F$ and $G$ violate the condition \textbf{\emph{(a)}} of Lemma \ref{eq_border}, they must also violate the condition  \textbf{\emph{(b)}} of this lemma, so we must have $ v = m(F, t + j - 1) \ge m(\overline{G}, j) = u$. Since $v$ is from $F$ and $u$ is not, we actually have $v > u$. Recall that $\F^*$ is closed under any left compression, in particular, under $\delta_{uv}$, which means that $\delta_{uv}(F) \in \F^*$.  On the other hand, since $v$ is in $F$ and $u$ is not, the weight of $\delta_{uv}(F)$ is smaller than the weight of $F$. Therefore a pair $(\delta_{uv}(F), G)$ must satisfy the condition \textbf{\emph{(a)}} of Lemma \ref{eq_border}. On the other hand, let us show that $(\delta_{uv}(F), G)$ violates the condition \textbf{\emph{(b)}} of Lemma \ref{eq_border}, and this will give us a contradiction. To do so, we show that for
 any $i\in [a - t + 1]$ the quantity $m(\overline{G}, i)$ does not exceed at least $a - t - i + 2$ different elements of $\delta_{uv}(F)$. This would mean that $m(\overline{G}, i) \le m(\delta_{uv}(F), t + i - 1)$ for every $i\in [a - t + 1]$, as required.

Recall that $F$ and $G$ violate the condition \textbf{\emph{(a)}} of Lemma \ref{eq_border} and hence also the condition \textbf{\emph{(b)}}. Thus $m(\overline{G}, i)$  does not exceed  $m(F, t + i - 1), \ldots, m(F, a)$. For $i > j$ all these elements are also in $\delta_{uv}(F)$, so in this case we are done. Now, for $i \le j$ among  $m(F, t + i - 1), \ldots, m(F, a)$ there is $v = m(F, t + j - 1)$ which is not in $\delta_{uv}(F)$, but instead in $\delta_{uv}(F)$ we have $u = m(\overline{G}, j) \ge m(\overline{G}, i)$. So if we remove $v$ from the list  $m(F, t + i - 1), \ldots, m(F, a)$ and add $u$ instead, all the elements of the list are still at least as large as $m(\overline{G}, i)$. Moreover, all the elements of the list are still distinct as $u$ is not from $F$.

%

\section{Proof of Proposition \ref{inf_counter}}
\label{sec:inf_counter}
By a discussion after the formulation of Proposition
  \ref{inf_counter} it is sufficient to show that for infinitely many
  $n$ with $\binom{n}{2}\equiv 0\pmod{2}$ there is no Hirschorn pair $\F,\G$ with $|\F| = |\G|$. We will show this for all 
 $n\equiv8\pmod{12}$. All such $n$ are divisible by $4$ so that $\binom{n}{2}$ is even.
For our setting of the parameters all the non-trivial Hirschorn pairs (i.e., Hirschorn pairs where both families are non-empty) look as follows:
$$
\F = \left\{F \in\binom{[n]}{2} : |F\cap [s]| \ge 1\}\right\}, \,\, \G = \left\{G \in\binom{[n]}{n - 2} : |G\cap [s]| \ge s\}\right\},$$
or as follows:
$$
\F = \left\{F \in\binom{[n]}{2} : |F\cap [s]| \ge 2\}\right\}, \,\, \G = \left\{G \in\binom{[n]}{n - 2} : |G\cap [s]| \ge s - 1\}\right\}.$$
In these two cases $|\F|, |\G|$ can be expressed as follows:
\begin{align*}
\mbox{\emph{(case 1):}  } |\F| &= s (n - s) + \binom{s}{2}, \qquad |\G| = \binom{n - s}{2}, \\
\mbox{\emph{(case 2):}  }|\F| &=  \binom{s}{2}, \qquad |\G| = s (n - s) + \binom{n - s}{2}. 
\end{align*}
We may consider only the first case as the second one reduces to the first one by a substitution $s\mapsto n - s$. Thus it is sufficient to demonstrate that the following equation in $s$ for all $n\equiv 8\pmod{12}$ has no integral solution:
$$s (n - s) + \frac{s(s - 1)}{2} = \frac{(n - s) (n - s - 1)}{2}.$$
By expanding the parentheses, we get:
$$2s^2 - (4n - 2) s + n^2 - n = 0.$$
Since  $n\equiv 8\pmod {12} \implies n\equiv 2\pmod{3}$,  we obtain:
$$2s^2 + 2  \equiv 0 \pmod{3} \implies s^2 \equiv 2\pmod{3},$$
but no such integral $s$ exists. 

\section{Proof of Proposition \ref{upper_bounds}}
\label{sec:upper_bounds}
\textbf{Proof of \emph{(a)}}. By Proposition \ref{M_prop} it is sufficient to show that 
$$\binom{s}{u} \binom{n - s}{a - u} \binom{s}{v} \binom{n - s}{b - v} \le \exp\left\{2h\left(\frac{a + b - 2t}{2n}\right)n\right\},$$
as long as $u + v \ge s + t$. Let us use the following abbreviations:
\begin{equation}
\label{abb}
\alpha = \frac{a}{n}, \,\, \beta = \frac{b}{n}, \,\, \tau = \frac{t}{n},\,\, \sigma = \frac{s}{n}, \,\, \mu = \frac{u}{n}, \,\, \eta = \frac{v}{n}.
\end{equation}
We will use the following upper and lower bounds on the binomial coefficients in terms of the Shannon function~\cite[Lemma 2.4.2]{cohen1997covering}:
\begin{equation}
\label{shannon}
 \frac{1}{\sqrt{8n}} \cdot  \exp\left\{h\left(\frac{k}{n}\right) \cdot n\right\}\le \binom{n}{k} \le \exp\left\{h\left(\frac{k}{n}\right) \cdot n\right\}.
\end{equation}

By \eqref{shannon} we only have to show that
$$ \sigma\left[h\left(\frac{\mu}{\sigma}\right) + h\left(\frac{\eta}{\sigma}\right)\right] + (1 - \sigma) \left[h\left(\frac{\alpha - \mu}{1 - \sigma}\right) + h\left(\frac{\beta - \eta}{1 - \sigma}\right)\right] \le 2h\left(\frac{\alpha + \beta}{2} - \tau\right), $$
as long as $\mu + \eta \ge \sigma + \tau$.

Due to concavity of $h$ 
we get:
\begin{align*}
&\sigma\left[h\left(\frac{\mu}{\sigma}\right) + h\left(\frac{\eta}{\sigma}\right)\right] + (1 - \sigma) \left[h\left(\frac{\alpha - \mu}{1 - \sigma}\right) + h\left(\frac{\beta - \eta}{1 - \sigma}\right)\right] \\
&\le 2\sigma h\left(\frac{\mu + \eta}{2\sigma}\right) + 2(1 - \sigma) h\left(\frac{\alpha - \mu + \beta - \eta}{2(1 -\sigma)}\right) \\
&= 2\sigma h\left(\frac{1}{2} + \frac{\mu + \eta - \sigma}{2\sigma}\right) + 2(1 - \sigma) h\left(\frac{1}{2} - \frac{\mu + \eta - \alpha - \beta + 1 - \sigma}{2(1 - \sigma)}\right).
\end{align*}
Due to the condition $\mu + \eta \ge \sigma + \tau$ we have
 $\mu + \eta - \sigma \ge \tau \ge 0$ and $\mu + \eta - \alpha - \beta + 1 - \sigma\ge \tau + 1 - \alpha - \beta$. In turn, due to the assumption $a + b < n + t$ we have that the quantity $\tau + 1 - \alpha - \beta$ is positive. This gives us the following:
\begin{align*}
&\sigma\left[h\left(\frac{\mu}{\sigma}\right) + h\left(\frac{\eta}{\sigma}\right)\right] + (1 - \sigma) \left[h\left(\frac{\alpha - \mu}{1 - \sigma}\right) + h\left(\frac{\beta - \eta}{1 - \sigma}\right)\right]\\
&\le 2\sigma h\left(\frac{1}{2} + \frac{\tau}{2\sigma}\right) + 2(1 - \sigma) h\left(\frac{1}{2} - \frac{\tau + 1 - \alpha - \beta}{2(1 - \sigma)}\right) \\
&= 2\sigma h\left(\frac{1}{2} - \frac{\tau}{2\sigma}\right) + 2(1 - \sigma) h\left(\frac{1}{2} - \frac{\tau + 1 - \alpha - \beta}{2(1 - \sigma)}\right)
\end{align*}

By using concavity of $h$ again we can bound the last expression as follows:
\begin{align*}
2h\left(\sigma\left[\frac{1}{2} - \frac{\tau}{2\sigma}\right] + (1 - \sigma)\left[\frac{1}{2} - \frac{\tau + 1 - \alpha - \beta}{2(1 - \sigma)}\right]\right) = 2h\left(\frac{\alpha + \beta}{2} - \tau\right).
\end{align*}

\bigskip

\textbf{Proof of \emph{(b)}}.  This upper bound takes the following form:
$$\frac{N_{prod}(n,a, b, t)}{8n^4 \cdot \binom{n}{a}\cdot\binom{n}{b}} \le \begin{cases}\exp\left\{-\frac{t^2}{8\left(\min\{a, n - a\} + \min\{b, n - b\}\right)}\right\} & a + b \le n, \\[10pt] \exp\left\{-\frac{(n - a - b + t)^2}{8\left(\min\{a, n - a\} + \min\{b, n - b\}\right)}\right\} &  a + b > n.\end{cases}$$
It is sufficient to establish this bound only in the case $a + b \le n$. Indeed, by taking the complements of the sets it is easy to obtain $N_{prod}(n, a, b, t) = N_{prod}(n, n - a, n - b, n - a - b + t)$.

So below we assume that $a + b \le n$. 
 By Proposition \ref{M_prop} it is sufficient to show that 
$$\frac{\binom{s}{u} \binom{n - s}{a - u} \binom{s}{v} \binom{n - s}{b - v}}{\binom{n}{a}\cdot \binom{n}{b}} \le 8n \cdot \exp\left\{-\frac{t^2}{8 (\min\{a, n - a\}  + \min\{b, n - b\})}\right\},$$
as long as $u + v \ge s + t$. We will use the same abbreviations as in \eqref{abb}. Due to \eqref{shannon} we only have to show that:

\begin{align*} &\sigma\left[h\left(\frac{\mu}{\sigma}\right) + h\left(\frac{\eta}{\sigma}\right)\right] + (1 - \sigma) \left[h\left(\frac{\alpha - \mu}{1 - \sigma}\right) + h\left(\frac{\beta - \eta}{1 - \sigma}\right)\right]  - h(\alpha) - h(\beta)
\\ &\le -\frac{\tau^2}{8(\min\{\alpha, 1 - \alpha\} + \min\{\beta, 1 - \beta\})},
\end{align*}
provided that $\mu + \eta \ge \sigma + \tau$.
The left-hand side of the last inequality is the sum of the following two quantities:
$$A = \sigma\cdot h\left(\frac{\mu}{\sigma}\right) + (1 - \sigma)\cdot h\left(\frac{\alpha - \mu}{1 - \sigma}\right) - h(\alpha),$$
$$B = \sigma\cdot h\left(\frac{\eta}{\sigma}\right) + (1 - \sigma)\cdot h\left(\frac{\beta - \eta}{1 - \sigma}\right) - h(\beta).$$
By concavity of $h$ both of these quantities are non-positive. So it is enough to show that either $A$ or $B$ does not exceed
$$\ -\frac{\tau^2}{8(\min\{\alpha, 1 - \alpha\} + \min\{\beta, 1 - \beta\})}.$$
Now, $\mu + \eta \ge \sigma + \tau \ge \sigma \alpha + \sigma \beta + \tau$ (the last inequality holds because we assume that $a + b\le n$). Hence either $\mu \ge \sigma\alpha + \tau/2$ or $\eta \ge \sigma\beta + \tau/2$. Let us show that:
\begin{align*}
\mu \ge \sigma\alpha + \tau/2 &\implies A \le  -\frac{\tau^2}{8(\min\{\alpha, 1 - \alpha\} + \min\{\beta, 1 - \beta\})}, \\
\eta \ge \sigma\beta + \tau/2 &\implies B\le  -\frac{\tau^2}{8(\min\{\alpha, 1 - \alpha\} + \min\{\beta, 1 - \beta\})}.
\end{align*}
We only show the first implication, 
the second one is proved similarly. 
Utilizing the Taylor series approximation for $h$, for some $\theta_1, \theta_2 \in [0, 1]$ we get:
\begin{align*}
A &= \sigma\left[h\left(\frac{\mu}{\sigma}\right) - h(\alpha) \right] + (1 - \sigma)\left[h\left(\frac{\alpha - \mu}{1 - \sigma}\right) - h(\alpha)\right] \\
&=\sigma\left[\left(\frac{\mu}{\sigma} - \alpha\right) h^\prime(\alpha) + \frac{\left(\frac{\mu}{\sigma} - \alpha\right)^2}{2} h^{\prime\prime}\left(\theta_1\alpha + (1 - \theta_1)\frac{\mu}{\sigma}\right) \right] + \\
&(1 -\sigma)\left[\left(\frac{\alpha - \mu}{1 - \sigma} - \alpha\right) h^\prime(\alpha) + \frac{\left(\frac{\alpha - \mu}{1 - \sigma} - \alpha\right)^2}{2} h^{\prime\prime}\left(\theta_2\alpha + (1 - \theta_2)\frac{\alpha - \mu}{1 - \sigma}\right) \right] \\
&= \frac{(\mu - \sigma\alpha)^2}{2}\left[\frac{h^{\prime\prime}\left(\theta_1\alpha + (1 - \theta_1)\frac{\mu}{\sigma}\right)}{\sigma} +\frac{ h^{\prime\prime}\left(\theta_2\alpha + (1 - \theta_2)\frac{\alpha - \mu}{1 - \sigma}\right)}{1 - \sigma}\right].
\end{align*}
Since we assume that  $\mu \ge \sigma\alpha + \tau/2$, we have:
$$\frac{\alpha - \mu}{1 - \sigma}\le \alpha \le \frac{\mu}{\sigma}.$$
 Now, it is easy to calculate $h^{\prime\prime}(x) = -\frac{1}{x} - \frac{1}{1 - x}$. Hence we can bound:
\begin{align*}
\frac{1}{\sigma} \cdot h^{\prime\prime}\left(\theta_1\alpha + (1 - \theta_1)\frac{\mu}{\sigma}\right) &\le -\frac{1}{1 - \theta_1\alpha - (1 - \theta_1)\frac{\mu}{\sigma}} \le -\frac{1}{1 -\alpha}, \\
\frac{1}{1 - \sigma} \cdot h^{\prime\prime}\left(\theta_2\alpha + (1 - \theta_2)\frac{\alpha - \mu}{1 - \sigma}\right) &\le -\frac{1}{\theta_2\alpha + (1 - \theta_2)\frac{\alpha - \mu}{1 - \sigma}} \le -\frac{1}{\alpha}.
\end{align*}
This yields our desired upper bound on $A$:
\begin{align*}
A&\le \frac{(\mu - \sigma\alpha)^2}{2} \cdot \left(-\frac{1}{\alpha} - \frac{1}{1 - \alpha}\right) \le -\frac{(\mu - \sigma\alpha)^2}{2} \cdot \frac{1}{\min\{\alpha, 1- \alpha\}} \\
&\le -\frac{(\mu - \sigma\alpha)^2}{2}\cdot \frac{1}{\min\{\alpha, 1- \alpha\} + \min\{\beta, 1 - \beta\}} \\
&\le -\frac{\tau^2}{8(\min\{\alpha, 1- \alpha\} + \min\{\beta, 1 - \beta\})}.
\end{align*}
(the last inequality is due to the assumption $\mu\ge \sigma\alpha + \tau/2$).

\paragraph{Acknowledgments.}
The work of the  first and third authors was done 
 within the framework of the HSE University Basic Research Program. The second author is supported by the EPSRC grant EP/P020992/1 (Solving Parity Games in Theory and Practice).
The third author was funded partially by the RFBR grant 20-01-00645.

\bibliography{extremal-combinatorics}

\end{document}